\documentclass[12pt]{amsart}
\usepackage[colorlinks=true,pagebackref,hyperindex,citecolor=Green,linkcolor=Blue]{hyperref}
\usepackage[top=1in, bottom=1in, left=1.1in, right=1.1in]{geometry}
\usepackage{amsmath}
\usepackage{amsfonts}
\usepackage{amssymb}
\usepackage{amsthm}
\usepackage{stmaryrd}
\usepackage[all]{xy}
\usepackage{mathrsfs}
\usepackage{enumitem} 
\usepackage[T1]{fontenc}
\usepackage{color}

\makeatletter
\def\namedlabel#1#2{\begingroup
#2%
\def\@currentlabel{#2}%
\phantomsection\label{#1}\endgroup
}
\makeatother

\theoremstyle{theorem} 
\newtheorem{theorem}{Theorem}[section]

\newtheorem{corollary}[theorem]{Corollary}

\newtheorem{lemma}[theorem]{Lemma}
\newtheorem{proposition}[theorem]{Proposition}

\newcommand{\Cech}{ \check{\rm{C}}}

\newtheorem{theoremx}{Theorem}
\theoremstyle{definition} 
\newtheorem{definition}[theorem]{Definition}
\newtheorem{question}[theorem]{Question}
\newtheorem{example}[theorem]{Example}
\newtheorem{remark}[theorem]{Remark}
\numberwithin{equation}{subsection}

\newcommand{\NN}{\mathbb{N}}
\newcommand{\RR}{\mathbb{R}}
\newcommand{\ZZ}{\mathbb{Z}}
\newcommand{\QQ}{\mathbb{Q}}
\newcommand{\FF}{\mathbb{F}}
\newcommand{\CC}{\mathbb{C}}

\newcommand{\m}{\mathfrak{m}}

\newcommand{\cR}{{\mathcal R}}
\newcommand{\cO}{{\mathcal O}}
\newcommand{\cC}{{\mathcal C}}

\makeatletter
\def\@tocline#1#2#3#4#5#6#7{\relax
  \ifnum #1>\c@tocdepth % then omit
  \else
    \par \addpenalty\@secpenalty\addvspace{#2}%
    \begingroup \hyphenpenalty\@M
    \@ifempty{#4}{%
      \@tempdima\csname r@tocindent\number#1\endcsname\relax
    }{%
      \@tempdima#4\relax
    }%
    \parindent\z@ \leftskip#3\relax \advance\leftskip\@tempdima\relax
    \rightskip\@pnumwidth plus4em \parfillskip-\@pnumwidth
    #5\leavevmode\hskip-\@tempdima
      \ifcase #1
       \or\or \hskip 1.9em \or \hskip 2em \else \hskip 3em \fi%
      #6\nobreak\relax
    \dotfill\hbox to\@pnumwidth{\@tocpagenum{#7}}\par
    \nobreak
    \endgroup
  \fi}
\makeatother

\newcommand{\Hom}{\operatorname{Hom}}

\newcommand{\Ker}{\operatorname{Ker}}
\newcommand{\CoKer}{\operatorname{Coker}}
\newcommand{\Ass}{\operatorname{Ass}}

\newcommand{\IM}{\operatorname{Im}}

\newcommand{\fpt}{\operatorname{fpt}}

\renewcommand{\a}{\mathfrak{a}}
\renewcommand{\b}{{J}}

\definecolor{blue-violet}{rgb}{0.54, 0.17, 0.89}
\definecolor{Blue}{rgb}{0.01, 0.28, 1.0}
\definecolor{gGreen}{rgb}{0.2, 0.8, 0.2}
\definecolor{Green}{rgb}{0.04, 0.85, 0.32}

\begin{document}

\title[$D$-modules, Bernstein-Sato polynomials $\&$ $F$-invariants]{$D$-modules, Bernstein-Sato polynomials and $F$-invariants of direct summands}

\author[J. \`Alvarez Montaner]{Josep \`Alvarez Montaner{$^1$}}
\address{Deptartament de Matem\`atiques\\
Universitat Polit\`ecnica de Catalunya\\ Av. Diagonal 647, Barcelona
08028, Spain} \email{Josep.Alvarez@upc.edu}

\author[C. Huneke]{Craig Huneke{$^2$}}
\address{Department of Mathematics, University of Virginia, Charlottesville, VA 22904-4135, USA}
\email{huneke@virginia.edu}

\author[L. N\'u\~nez-Betancourt]{Luis N\'u\~nez-Betancourt${^3}$}
\address{Centro de Investigaci\'on en Matem\'aticas, Guanajuato, Gto., M\'exico}
\email{luisnub@cimat.mx}

\thanks{{$^1$}The first author was partially supported by Generalitat de Catalunya 2014SGR-634 project
and Spanish Ministerio de Econom\'ia y Competitividad
MTM2015-69135-P}
\thanks{{$^2$}The second author was partially supported by the NSF Grant 1460638}
\thanks{{$^3$}The third author was partially supported by the NSF Grant 1502282 }

\subjclass[2010]{Primary 	14F10,	13N10, 13A35, 16S32   ; Secondary 13D45, 14B05, 	14M25, 13A50 .}
\keywords{$D$-modules; Bernstein-Sato polynomial; direct summands; ring of invariants;  local cohomology;  $F$-jumping numbers; test ideals; $F$-thresholds.}

\maketitle
\begin{abstract}
We study the structure of $D$-modules over a ring $R$  which is a direct summand of a polynomial or a
power series ring $S$ with coefficients over a field.
We relate properties of $D$-modules over $R$ to $D$-modules over $S$.
 We show that the localization $R_f$ and the local cohomology module $H^i_I(R)$ have finite length as $D$-modules over $R$.  
 Furthermore, we show the existence of the Bernstein-Sato polynomial for elements
  in $R$. In positive characteristic, we use this relation between $D$-modules over
 $R$ and $S$ to show that the set of $F$-jumping numbers of an ideal $I\subseteq R$ is contained in the set of $F$-jumping numbers of its extension in $S$. As a consequence, the $F$-jumping numbers of  $I$ in $R$ form a discrete set of rational numbers. 
 We also relate the Bernstein-Sato polynomial in $R$ with the $F$-thresholds and the $F$-jumping numbers in $R$.
\end{abstract}
\tableofcontents

%%%%%%%%%%%%%%%%%%%%%%%%%%%%%%%%%%%%%%%%%%%%%%%%%%%%%%%%%%%%%
\section{Introduction}
%%%%%%%%%%%%%%%%%%%%%%%%%%%%%%%%%%%%%%%%%%%%%%%%%%%%%%%%%%%%%

Direct summands of polynomial rings play an important role in
representation theory, combinatorial algebra, commutative algebra,
and algebraic geometry.  For instance, rings of invariants under a
linearly reductive group actions, and  affine toric rings are direct
summands of polynomial rings. Rings corresponding
to Grassmannian,  Veronese and Segre  varieties also belong to this
family of rings. If one assumes that the field has characteristic
zero, then the ring associated to the $n\times m$ matrices of rank
at most $t$ is a direct summand of a polynomial ring. Since the
Hochster-Roberts Theorem \cite{HoRo}, which states that direct
summands of regular rings are Cohen-Macaulay, it has been expected
that direct summands behave similarly to regular rings
in certain aspects. 
For instance, direct summands are normal rings \cite{HHStrongFreg} with rational singularities 
in characteristic zero \cite{Boutot} or strongly $F$-regular singularities in prime characteristic \cite{HHStrongFreg}.
 The main goal of this manuscript is to find structural properties of
important $D$-modules over direct summands of regular rings, which
resemble the situation over regular rings.

We first focus on modules over the ring of $K$-linear differential
operators, $D_{R|K}$, where $K$ is a field, and $R$ is a direct summand of a regular
$K$-algebra $S$.  
Significant  research has been focused on 
the finite generation of $D_{R|K}$ for rings of invariants
\cite{Kantor,LS,Levasseur,TravesOrbifolds,Schwarz,MussonVdB,LSDiffOpInv,PleskenRobertz},
Grassmann varieties \cite{TravesGrass}, and  toric affine algebras
\cite{Musson,Jones,MussonJPPA,SaitoTraves,SaitoTraves2,SaitoTakahashi}.
In addition, the  simplicity of $D_{R|K}$ as a ring and the
simplicity of $R$ as $D_{R|K}$-module have also been a topic of
interest \cite{SmithVDB,DModFSplit,TravesHS}. Unfortunately, the
rings of differential operators are not always finitely generated
for singular rings, even in characteristic zero \cite{DiffNonNoeth}.
Instead of focusing on the ring $D_{R|K}$, we switch gears and study
the structure of an important class of  $D_{R|K}$-modules.

We  study the structure of localization and local cohomology modules
over direct summands. For this we recall their finiteness properties
over regular rings. Suppose that $S$ is either  $K[x_1,\ldots,x_n]$
or $K[[x_1,\ldots,x_n]]$. If $K$ has characteristic zero, then the
ring of $K$-linear differential operators,  $D_{S|K},$ equals the
$S$-algebra generated by the partial derivatives, $S\left\langle
\frac{\partial}{\partial_{x_1}}, \ldots,
\frac{\partial}{\partial_{x_n}}\right\rangle$. In particular,
$D_{S|K}$ is a left and right Noetherian ring.  Furthermore, every
localization $S_f$, and therefore every local cohomology module
$H^i_I(S)$, has finite length as $D_{S|K}$-module. If $K$ has prime
characteristic,   $D_{S|K}$ is no longer Noetherian. However, $S_f$
and $H^i_I(S)$ still have finite length as $D_{S|K}$-modules. These
finiteness  properties have a significant impact on the study of
local cohomology over regular rings \cite{LyuDMod,LyuFmod}. Our
first result recovers these properties for direct summands.

\begin{theoremx}[{see Theorem \ref{ThmFinLenLC}}]\label{MainLen}
Let $K$ be a field and $S$ be either  $K[x_1,\ldots,x_n]$ or $K[[x_1,\ldots,x_n]]$, and $R$ be a $K$-subalgebra.
Suppose that $R$ is a direct summand  of $S$.
Then $R_f$ and $H^i_I(R)$ have finite length as $D_{R|K}$-modules for every $f\in R$ and any ideal $I\subseteq R.$
Furthermore,
$$
\lambda_{D_{R|K}}(R_f)\leq \lambda_{D_{S|K}}(S_f)
\hbox{ and }\lambda_{D_{R|K}}(H^i_I(R))\leq \lambda_{D_{S|K}}(H^i_{IS}(S)).
$$
\end{theoremx}

The previous theorem holds for a larger class of modules (see Theorem \ref{ThmFinLenC(R,K)}),
which includes $H^{i_1}_{I_1}(\cdots (H^{i_\ell}_{I_\ell}(R))\cdots).$
Lyubeznik introduced a category of
$D_{R|K}$-modules  \cite{Lyu2,LyuUMC}, denoted by $C(R,K)$, to compensate for the lack of the
notion of holonomic $D$-modules in prime and mixed characteristic at the time.
The category $C(R,K)$ resembles the class of holonomic modules in
the sense that both consist of $D$-modules of finite length in equal characteristic.

Since its introduction \cite{BernsteinPoly,SatoPoly}, the
Bernstein-Sato polynomial $b_f(s)$ has been an important invariant of a hypersurface in characteristic  zero. In fact,  $b_f(s)$  is an important object in
the study of singularities. For instance, it relates to the theory of
vanishing cycles \cite{Del}, $V$-filtrations  and monodromies
\cite{Malgrange,KashiwaraV}, jumping numbers of multiplier ideals
\cite{Kol,ELSV,BudurSaito}, and zeta functions \cite{DenefLoeser}. In
our second main result, we develop the theory of the Bernstein-Sato
polynomial over a direct summand.  
To the best of our knowledge,
this is one of the first efforts to extend the theory of the Bernstein-Sato polynomial for singular rings 
(see  \cite{HsiaoMatusevich} for an study of Bernstein-Sato polynomials of ideals  in a normal toric ring).

\begin{theoremx}[{see Theorem \ref{ThmGralBShyp}}]\label{MainBSpoly}
Let $K$ be a field of characteristic zero, let $S$ be either
$K[x_1,\ldots,x_n]$ or $K[[x_1,\ldots,x_n]]$, and let $R$  be a
$K$-subalgebra. Suppose that $R$ is a direct summand  of $S$. Then for
every element $f\in R\setminus\{0\}$, there exists $\delta(s)\in
D_{R|K}[s]$ and $b(s)\in \QQ[s]$ such that
\begin{equation}\label{EqThmBS}
\delta(t)\cdot f^{t+1}=b(t)f^t
\end{equation}
for every $t\in\ZZ.$
As a consequence, $R_f$ is a cyclic $D_{R|K}$-module.
\end{theoremx}

Under the hypothesis of the previous theorem, we call the Bernstein-Sato polynomial of $f$ in $R$, $b^R_f(s)$,
the monic polynomial of smallest
degree satisfying the Equation \ref{EqThmBS} for some $\delta(s)$.
 As a consequence of the proof of Theorem \ref{MainBSpoly}, the roots of $b^R_f(s)$ consists 
of negative rational numbers when $R$ is a direct summand of a polynomial ring. %\cite{Malgrange2,Kashiwara}.} 
We hope that $b^R_f(s)$  relates to other invariants that measure singularities of hypersurfaces in direct summands.

In Remark \ref{RmkGralBS}, we show how different versions 
of the Bernstein-Sato theory can be extended to direct summands, which include  
the one given by Sabbah  \cite{Sab} and the one given by Budur,  Musta{\c{t}}{\v{a}}, and Saito \cite{BMSBS}.
Very recently, Hsiao and  Matusevish extended and studied Bernstein-Sato polynomials
associated to ideals in a normal toric ring \cite{HsiaoMatusevich}.
We note that not every $K$-algebra has  Bernstein-Sato polynomials (see Example \ref{ExampleNoBS}).

We point out that it was previously known that if $R$ is an affine
toric ring and $I\subseteq R$  is a monomial ideal, then $H^i_I(R)$ has finite length as $D_{R|K}$ module
\cite{Hsiao}. In this case, it was also shown that if $f$ is a
monomial, then $R_f$ is a cyclic $D_{R|K}$-module \cite{Hsiao}.  
Even for toric rings, Theorems \ref{MainLen} and \ref{MainBSpoly} recover
and extend these results, as the monomial hypothesis is not longer needed.

In the second part of the paper we shift to positive characteristic to study the ring of
differential operators, $D_{R}$. We show that $R_f$ is generated
by $\frac{1}{f}$ as $D_{R}$-module for direct summands of regular
rings (see Proposition \ref{PropLocCyclicPrime}).
% This extends previous results know for regular rings \cite{AMBL}, and for rings with finite $F$-representation type \cite{TT}.
We use this fact and the ideas behind the proof of Theorems
\ref{MainLen} and \ref{MainBSpoly} to study the $F$-jumping numbers
of $R$. These invariants are used to measure singularities in prime characteristic. Specifically, the $F$-jumping numbers
are the values where the generalized test ideals change \cite{H-Y}.
 The test ideals can be seen as an
analogue, in prime characteristic, of multiplier ideals
\cite{KarenFrational,KarenComm,H-Y}. In particular, the $F$-pure
threshold, which is the first $F$-jumping number, serves as the
analogue of the log-canonical threshold \cite{H-Y,MTW}.
We recall that multiplier ideals are
defined using %Hironaka's  \cite{Hironaka}
resolution of singularities (see \cite{LazBook2}). In contrast, test
ideals are defined in terms of tight closure theory
\cite{HoHu1,HoHu2,HoHu3,H-Y} so, a priori, it is not clear that
$F$-jumping numbers form a discrete set of rational numbers. Since
the introduction of the test ideals \cite{H-Y}, intense research has
been devoted around this question
\cite{BMS-MMJ,BMS-Hyp,TestQGor,BSTZ,ST-NonPrincipal,CEMS}.
%This task is differs from the analogue in characteristic zero as currently
%there is no resolution of singularities in prime characteristic.
Efforts have also been dedicated to compute these invariants, mainly
over regular rings
\cite{Binomials,ST-NonPrincipal,Diagonals,QuasiHomog}.  Our main
result in this direction is that the set of the $F$-jumping numbers of
$R$ is a subset of $F$-jumping numbers of $S$.

\begin{theoremx}[{see Theorem \ref{ThmEqTestIdeals} and Corollary \ref{CorMainDS}}]\label{MainFjump}
Let $S$  be a regular $F$-finite domain, and $R$ be an $F$-finite ring. 
Suppose that $R\subseteq S$ and that $R$ is a direct summand of $S$. Let $I\subseteq R$ denote an ideal. Then, the set
$F$-jumping numbers of $I$ in $R$ is a subset of the set
$F$-jumping numbers of $I S$ in $S$. In particular, the set
$F$-jumping numbers of $I$ in $R$ is formed by rational  numbers
and has no accumulation points. As a consequence, the $F$-pure
threshold of $I$ is a rational number.
\end{theoremx}

The previous result is relevant when the extension $R\to S$ is not finite, since  the finite case follows from the work of Schwede and Tucker \cite{STFiniteMaps}.

Theorem \ref{MainFjump} gives a set of candidates for the $F$-jumping
numbers in $R$, which could potentially give algorithms to compute the
invariants for direct summands (see Remark \ref{RemAlg}).

Motivated by the work of Musta{\c{t}}{\v{a}}, Takagi and Watanabe
\cite{MTW} for smooth complex varieties, we relate our notion of Bernstein-Sato polynomial to invariants in
positive characteristic, namely the $F$-thresholds and the
$F$-jumping numbers. These invariants are equal for regular rings
\cite{BMS-MMJ}; however, they may differ for singular rings. Suppose that $R$ is a direct summand of a polynomial ring over $\QQ$, and $f\in R$. In
Theorems \ref{ThmBSFThresholds} and \ref{ThmBSFjumps} we show that
the truncated base $p$-expansion of any $F$-jumping number and any
$F$-threshold of the reduction of $f$ modulo $p$ is a root of the
Bernstein-Sato polynomial  $b_f(s)$ modulo $p.$ These  type of results have
been used to recover roots of the Bernstein-Sato polynomials via
methods in prime characteristic  in the regular setting (see \cite{BMS-MRL}).

\section{Background on $D$-modules} \label{Dmod}

In this section we briefly recall the basics on the theory of rings
of differential operators as introduced  by Grothendieck \cite[\S 16.8]{EGA}.

\vskip 2mm

 Let $R$ be a Noetherian ring. The ring of differential
operators of $R$ is the subring $D_{R}\subseteq \Hom_{\ZZ}(R,R)$
whose elements are defined inductively as follows:  the differential
operators of order zero are defined by the multiplication by
elements of $R$, i.e. $D^{0}_{R}\cong R.$ We say that
$\delta\in \Hom_\ZZ(R,R)$ is an operator of order less than or equal
to $m$ if $[\delta,r]=\delta r-r\delta$ is an operator of order less
than or equal to $m-1.$ We  have a filtration $D^{0}_{R} \subseteq
D^{1}_{R} \subseteq \cdots$ given by the order and the ring of
differential operators is defined as
$$D_{R}=\bigcup_{m\in\NN}D^{m}_{R}.$$  If $A\subseteq R$ is a subring, we
denote $D_{R|A} \subseteq D_R$ the subring of differential operators
that are $A$-linear. In particular $D_R=D_{R|\ZZ}$. If $R$
contains a field $K$ of characteristic $p>0$,  every
additive map is $Z/p\mathbb Z$-linear and thus $D_R=D_{R|\mathbb
Z/p\mathbb Z}$. Moreover, let $R^{p^e} \subseteq R$ be the subring
consisting of $p^{e}$ powers of elements of $R$ and set
$D^{(e)}_{R}:= \Hom_{R^{p^e}}(R,R)$. In this case,
$D_{R|K} \subseteq D_{R} \subseteq \bigcup_{e\in\NN}D^{(e)}_{R}$.
If $K$ is a perfect field, then $D_{R|K} = D_{R}$ and $D_{R} =
\bigcup_{e\in\NN}D^{(e)}_{R}$  when $R$ is $F$-finite (see
 \cite[Theorem 2.7]{SmithSP} and \cite[Theorem 1.4.9]{Ye}).

\begin{example}
Let $R$ be either the polynomial ring $A[x_1,\dots,x_n] $ or the formal power
series ring $A[[x_1,\dots,x_n]]$ with coefficients in a ring $A$. The
ring of $A$-linear differential operators is:
$$D_{R|A}= R \left \langle \hskip 2mm \frac{1}{t!}\frac{d^t}{dx_i^t} \hskip 2mm | \hskip 2mm i=1,\dots,n; \hskip 2mm t\in \NN \right\rangle ,$$
that is, the free $R$-module generated by the differential operators $\frac{1}{t!}\frac{d^t}{dx_i^t}$.
Furthermore, if $A=K$ is a field of characteristic zero, we have
 $$D_{R|K}= R \left \langle  \frac{d}{dx_1},\dots, \frac{d}{dx_n} \right\rangle.$$
\end{example}

\begin{example}
Let $S= K[x_1,\dots,x_n]$ be a polynomial ring over a field $K$ of
characteristic zero and, given an ideal $I\subseteq S$, set $R=S/I$.
Then, using the results about differential operator of quotient rings of polynomials \cite[Theorem 5.13]{MCR}, the ring of $K$-linear
differential operators of $R$ is given by $$D_{R|K}:=
\frac{\{ \delta \in D_{S|K} \hskip 2mm | \hskip 2mm \delta (I)
\subseteq I \}}{ID_{S|K}}$$
\end{example}

It is clear that $R$ is a $D_{R|A}$ module. Several of the results presented in this manuscript concern to the
$D_{R|A}$-module structure of the localization, $R_f$ at an element $f\in R$, and the local
cohomology modules $H^i_I(R)$ associated to an ideal $I\subseteq R$.
For this reason, we recall a few properties
and definitions regarding these objects.

\vskip 2mm

$\bullet$ {\it Localization:} Let $M$ be a $D_{R|A}$- module and $f\in
R$. Then $M_f$ is also a $D_{R|A}$-module, where the action of a
differential operator $\delta\in D_{R|A}$ on $\frac{v}{f^t}\in M_f$
is defined inductively as follows: If $\delta \in D_{R|A}^0$ has
order zero then $\delta \cdot \frac{v}{f^t}=\frac{ \delta\cdot
v}{f^t}.$ Now, suppose that the action of every element in
$D_{R|A}^n$ has been defined. Let $\delta\in D^{n+1}_{R|A}.$ Then,
$$
\delta \cdot \frac{v}{f^t} = \frac{\delta\cdot v - [\delta,f^t]\cdot
\frac{v}{f^t}}{f^t},
$$
which  follows from the fact that $[\delta,f^t]\cdot \frac{v}{f^t}$
is defined. With this $D_{R|A}$-module structure on $M_f$, the
localization map $M\to M_f$ is a morphism of $D_{R|A}$-modules.

\vskip 2mm

$\bullet$ {\it Local cohomology modules:} Let $I\subseteq R$ be an ideal generated by a sequence
of elements $\underline{f}=f_1,\ldots,f_\ell\in R$,  and let $M$ be any
$R$-module. The $\check{\mbox{C}}$ech complex of $M$ with respect to
$\underline{f}$ is defined by
$$
\Cech^\bullet(\underline{f};M): \hskip 3mm 0\to M\to \bigoplus_i
M_{f_i}\to\bigoplus_{i,j} M_{f_i f_j}\to \cdots \to M_{f_1 \cdots
f_\ell} \to 0,
$$
where the maps on every summand are localization maps up to a sign.
The local cohomology of $M$ with support on $I$ is defined by
$$
H^i_I(M)=H^i(\Cech^\bullet(\underline{f};M)).
$$
One may check that it is independent of the set of generators of $I$.
It follows from this construction that every local cohomology
module over a $D_{R|A}$-module is again a $D_{R|A}$-module.

%%%%%%%%%%%%%%%%%%%%%%%%%%%%%%%%%%%%%%%%%%%%%%%%%%%%%%%%%%%%%
\section{D-modules over direct summands} \label{Dmod_summand}
%%%%%%%%%%%%%%%%%%%%%%%%%%%%%%%%%%%%%%%%%%%%%%%%%%%%%%%%%%%%%%

Let $A\subseteq R\subseteq S$ be an extension of Noetherian
rings such that $R$ is a direct summand of $S$ with splitting
morphism $\beta:S\to R$. Our first aim in this section is to relate
differential operators on $S$ to differential operators on $R$ using
the splitting $\beta$. Then we  introduce  $D_{R|A}$-modules
that are direct summands of $D_{S|A}$-modules,  for which the
differential structure is compatible. This idea is used to obtain
properties of the localizations $R_f$ and the local cohomology
modules $H^i_I(R)$ as a $D_{R|A}$-module from the properties of
$S_f$ and $H^i_{IS}(S)$ as a $D_{S|A}$-module. The same techniques
are used in order to develop a theory of Bernstein-Sato polynomials, even if $R$ is not regular.

%%%%%%%%%%%%%%%%%%%%%%%%%%%%%%%%%%%%%%%%%%%%%%%%%%%%%%%%%%%%%%
\subsection{Differentiable direct summands} \label{Diff_summand}
%%%%%%%%%%%%%%%%%%%%%%%%%%%%%%%%%%%%%%%%%%%%%%%%%%%%%%%%%%%%%%

Let $A\subseteq R\subseteq S$ be Noetherian rings such that $R$
is a direct summand of $S$ with splitting morphism $\beta:S\to R$.
We start relating differential operators in $D_{S|A}$ to
differential operators in $D_{R|A}$  using the splitting $\beta$. We first show that the order of a differential operator cannot
increase when  we compose it with the splitting $\beta$.  This idea
has been used  for rings of invariants over a group action
\cite{Schwarz}. Namely, there exists a map\footnote{$D_{S|K}^G$
denotes the ring of equivariant differential operators over the field $K$.} $D_{S|K}^G\to D_{S^G|K}. $ However, this map may not be
surjective \cite[Example 5.7]{Schwarz}.  
It is injective  for many classes of groups, in particular for finite groups, but, to the best of our knowledge, the injectivity of
this map is not known in general. 
For the study of local cohomology
modules of $S^G$ in characteristic zero as $D_{S|K}^G$-module, we
refer to \cite{PutInv}.

We point that the following lemma was implicit in Smith's proof of the statement: $S$ is a simple $D_{S|\FF_p}$-module, then $R$ is also $D_{R|\FF_p}$ is also simple \cite[Proposition 3.1]{DModFSplit}. We include this proof for the sake of completeness.

\begin{lemma}\label{LemmaRestriction}
Let $A\subseteq R\subseteq S$  be three Noetherian rings. Let
$\iota:R\to S$ denote the inclusion and let $\beta: S\to R$ be any
$R$-linear morphism. Then, for every $\delta\in D^n_{S|A}$, we have
that $\tilde{\delta}:=\beta\circ\delta_{|_{R}}\in D^n_{R|A}.$
\end{lemma}
\begin{proof}
To avoid heavy notation, we also write $\delta$ for $\delta_{|_{R}}$
We first note that $\tilde{\delta}\in \Hom_A(R,R).$ We proceed by
induction on the order $n$ of the differential operator. If $n=0,$
there exists $g\in S$ such that $\delta(w)=gw$ for  every $w\in S.$
Then,
$$
\tilde{\delta}(v)=\beta\circ\delta(v)=\beta(\delta(v))=\beta(gv)=v\beta(g)=\beta(g)
v.$$ for every $v\in R$. We now assume the claim for $n$. Let
$\delta\in D^{n+1}_{S|A}$, and $f\in R$. Then,
\begin{align*}
[\tilde{\delta},f](v)&=\tilde{\delta}(fv)-f\tilde{\delta}(v)\\
&=\beta\circ\delta(fv)-f(\beta\circ\delta(v))\\
&=\beta(\delta((fv))-\beta(f\delta(v))\\
&=\beta(\delta((fv)-f\delta(v))\\
&=\beta([\delta, f](v))\\
&=\widetilde{[\delta,f]}(v)
\end{align*}
for all $v\in R.$ Then,
$[\tilde{\delta},f]=\widetilde{[\delta,f]}\in D^n_{R|A}$ by induction
hypothesis. Then, $\tilde{\delta}\in D^{n+1}_{S|A}.$
\end{proof}

We now introduce  $D_{R|A}$-modules that are direct summands of
$D_{S|A}$-modules,  for which the differential structure is
compatible.

\begin{definition}
Let $A\subseteq R\subseteq S$ let be three Noetherian rings. Suppose
that the inclusion $R\subseteq S$ has a splitting $\beta:S\to R$.
We say that a $D_{R|A}$-module $M$, is a \emph{differential direct
summand of the $D_{S|A}$-module $N$ compatible with $\beta$}, if
$M\subseteq N $ and we have a splitting $\theta: N\to M$ of
$R$-modules such that
$$
\theta (\delta \cdot v)=(\beta\circ\delta_{|_{R}})\cdot v.
$$
for every $v\in M$ and $\delta\in D_{S|A}.$ In this case, we say that
$\theta$ is a differentiable splitting map compatible with $\beta.$
\end{definition}

\begin{lemma}\label{LemmaEqSubmod}
Let $A\subseteq R\subseteq S$  be three Noetherian rings. Suppose
that the inclusion $R\subseteq S$ has a splitting $\beta:S\to R$.
Let $M\subseteq N$ be differentiable direct summand compatible with
$\beta$. Let $V,W\subseteq M$  be $D_{R|A}$-submodules of $M$. If
$D_{S|A}V=D_{S|A}W$, then $V=W.$
\end{lemma}
\begin{proof}
We denote by $\theta:N\to M$ a differential splitting, and
$\beta\circ \delta_{|_{R}}$ by $\widetilde{\delta}$. Let $v\in V$.
Since $D_{S|A}V=D_{S|A}W,$ there exists
$\delta_1,\ldots,\delta_\ell\in D_{S|A}$ and $w_1,\ldots,w_\ell \in
W$ such that $v=\delta_1 w_1+\ldots +\delta_\ell w_\ell$. We apply
$\theta$ to both sides to obtain
$$
v  =\theta(v)
=\theta\left(\delta_1 w_1+\ldots +\delta_\ell w_\ell\right)
=\theta(\delta_1 w_1)+\ldots + \theta(\delta_\ell w_\ell)
=\widetilde{\delta_1} w_1+\ldots +\widetilde{\delta_\ell} w_\ell,
$$
where the last step follows from the fact that $\theta$ is a
differential splitting compatible with $\beta.$ Then, $v\in D_{R|A}
W=W.$ We conclude that $V\subseteq W.$ Likewise, $W\subseteq V,$
which concludes the proof.
\end{proof}

The main result of this subsection is the following upper bound for
the length, as $D$-module, of a differential direct summand.

\begin{proposition}\label{PropFinLenght}
Let $A\subseteq R\subseteq S$ let be three Noetherian rings. Suppose
that the inclusion $R\subseteq S$ has a splitting $\beta:S\to R$.
Let $M\subseteq N$ be differentiable direct summands compatible with
$\beta$. Then, $\lambda_{D_{R|A}}(M)\leq \lambda_{D_{S|A}}(N).$
\end{proposition}
\begin{proof}
We note that if $N$ does not have finite length as $D_{S|A}$-modules
the claim is clear. We assume that $\lambda_{D_{S|A}}(N)$ is finite,
and proceed by contradiction. Let $\ell=\lambda_{D_{S|A}}(N).$ Let
$0\subsetneq V_1\subsetneq\ldots\subsetneq V_{\ell+1}$ be a strictly
increasing chain of $D_{R|A}$-submodules of $M$. By Lemma
\ref{LemmaEqSubmod} $0\subsetneq D_{S|A}V_1\subsetneq\ldots\subsetneq
D_{S|A}V_{\ell+1}$ is a strictly increasing chain of
$D_{S|A}$-submodules of $N$. Then, $\ell+1\leq \lambda_{D_{S|A}}(N),$
which is a contradiction.
\end{proof}

\begin{definition}\label{DefDiffMorphism}
Let $A\subseteq R\subseteq S$ let be three Noetherian rings. Suppose
that the inclusion $R\subseteq S$ has a splitting $\beta:S\to R$.
Given two differentiable direct summands $M_1\subseteq N_1$ and
$M_2\subseteq N_2$ with differentiable splittings $\theta_1: N_1\to
M_1$ and $\theta_2: N_2\to M_2$, we say that a map $\phi:M_1\to M_2$
is a morphism of differential direct summands if $\phi\in
\Hom_{D_{S|A}}(N_1,N_2)$, $\phi(M_1)\subseteq M_2$,
$\phi_{|_{M_1}}\in \Hom_{D_{R|A}}(M_1,M_2)$, and the following
diagram
\begin{center}
$\xymatrix{
M_1 \ar[d]^{\phi_{|_{M_1}}} \ar[r] & N_1\ar[d]^{\phi} \ar[r]^{\theta_1} & M_1\ar[d]^{\phi_{|_{M_1}}}\\
M_2 \ar[r]  & N_2\ar[r]^{\theta_2} & M_2}$
\end{center}
commutes.
For the sake of notation, we often write $\phi$ for $\phi_{|_{N_1}}.$
\end{definition}
\begin{proposition}\label{PropLoc}
Let $A\subseteq R\subseteq S$ let be three Noetherian rings. Suppose
that the inclusion $R\subseteq S$ has a splitting $\beta:S\to R$.
Let  $M$ be a $D_{R|A}$-module and $N$  be a $D_{S|A}$-module. If $M$
is a differential direct summand of $N$, then $M_f$ is a differential
direct summand of $N_f$ for every element of $f\in R.$ Furthermore,
the localization maps give a morphism of differential direct
summands.
\end{proposition}
\begin{proof}
Let $\theta:N\to M$ be a differential splitting compatible with
$\beta$. For the sake of notation, we use the same symbols for
induced maps in the localization. We now show that the map $ M_f\to
N_f$ induced by $\theta$ is a differential splitting compatible
with $\beta$.

Let $\widetilde{\delta}$ denote the differential operator
$\beta\circ \delta_{|_{R}}$. We start showing that $\theta (\delta
\cdot \frac{v}{f^n})=\widetilde{\delta}\cdot \frac{v}{f^n}$  for
every element $v\in M$ and  $\delta\in D_{S|A},$ by induction on the
order of $\delta$. If $\delta$ has order zero we have, using the
action of $\delta$ as given in Section \ref{Dmod}, that
$$
\theta \left(\delta  \cdot \frac{v}{f^n}\right)= \theta
\left(\frac{\delta\cdot v}{f^n}\right)=\frac{1}{f^n}
\theta(\delta\cdot v) =\frac{1}{f^n} \theta(\delta\cdot v)
=\frac{1}{f^n}(\widetilde{\delta} \cdot v),$$ where the last
equality follows because $\theta:M\to N$ is a differentiable
splitting. Then,
\[
\theta \left(\delta  \cdot \frac{v}{f^n}\right)=
\frac{1}{f^n}(\widetilde{\delta} \cdot v)=\frac{\widetilde{\delta}
\cdot v}{f^n}= \widetilde{\delta}\cdot \frac{v}{f^n}.
\]
We now assume our claim for differential operators of order equal to
or less than $n$, and suppose that $\delta$ has order $n+1.$
\begin{align*}
\theta\left(\delta \cdot \frac{v}{f^t} \right)&=\theta\left(\frac{\delta\cdot v - [\delta,f^n]\cdot \frac{v}{f^t}}{f^t}\right)\\
&=\theta\left(\frac{\delta\cdot v }{f^t}\right)-
\theta\left(\frac{[\delta,f^n]\cdot \frac{v}{f^t}}{f^t}\right)\\
&=\frac{1}{f^t}\theta\left(\delta\cdot v \right)-
\theta\left(\frac{[\delta,f^n]\cdot \frac{v}{f^t}}{f^t}\right)\\
&=\frac{1}{f^t}\widetilde{\delta}\cdot v-
\theta\left(\frac{[\delta,f^n]\cdot \frac{v}{f^t}}{f^t}\right)\hbox{ because }\theta:M\to N\hbox{ is a differentiable spliting}\\
&=\frac{1}{f^t} \widetilde{\delta}\cdot v -
\frac{\widetilde{[\delta,f^n]}\cdot \frac{v}{f^t}}{f^t}\hbox{ by induction hypothesis as }[\delta,f^n]\in D^n_A(S);\\
&=\frac{\widetilde{\delta}\cdot v - \widetilde{[\delta,f^n]}\cdot \frac{v}{f^t}}{f^t}\\
&=\frac{\widetilde{\delta}\cdot v - \widetilde{[\delta,f^n]}\cdot \frac{v}{f^t}}{f^t}\\
&=\frac{\widetilde{\delta}\cdot v - [\widetilde{\delta},f^n]\cdot \frac{v}{f^t}}{f^t} \hbox{ as a consequence of the proof of Lemma \ref{LemmaRestriction}};\\
&=\widetilde{\delta} \cdot \frac{v}{f^t}
\end{align*}
To verify that the localization map $M\to M_f$ induces a map of
differential direct summands, we note that all the conditions in
Definition \ref{DefDiffMorphism} are satisfied by the usual
properties of localization maps, and the fact that localization maps
are $D$-module morphisms.
\end{proof}

\begin{lemma}\label{LemmaAb}
Let $A\subseteq R\subseteq S$ let be three Noetherian rings. Suppose
that the inclusion $R\subseteq S$ has a splitting $\beta:S\to R$.
Let $M_1\subseteq N_1$ and $M_2\subseteq N_2$ be differentiable
direct summands with differential splittings $\theta_1$ and
$\theta_2$. Let $\phi:N_1\to N_2$ be a map of differentiable direct
summands. Then, $\Ker(\phi_{|_{M_1}})\subseteq \Ker(\phi)$,
$\IM(\phi_{|_{M_1}})\subseteq \IM(\phi_{|_{M_1}})$ and
$\CoKer(\phi_{|_{M_1}})\subseteq \CoKer(\phi)$ are also
differentiable direct summands. Furthermore, the inclusion and
projection maps are morphism of differentiable direct summands.
\end{lemma}
\begin{proof}
Since $\phi$ is a map of $D_{S|A}$ modules and $\phi_{|_{M_1}}$ is
also a map of $D_{R|A}$-modules, we note that the induced maps on
kernel, images and cokernel are also maps of $D_{S|A}$ and
$D_{R|A}$-modules respectively. The rest follows from the fact that
the diagram
\begin{center}
$\xymatrix{
M_1 \ar[d]^{\phi_{|_{M_1}}} \ar[r] & N_1\ar[d]^{\phi} \ar[r]^{\theta_1} & M_1\ar[d]^{\phi_{|_{M_1}}}\\
M_2 \ar[r]  & N_2\ar[r]^{\theta_2} & M_2}.$
\end{center}
commutes, and that $\theta_1,\theta_2$ are differential splittings.
\end{proof}

\subsection{Finite length of $D$-modules over direct summands}
Let $K$ be a field and $S$ be either $K[x_1,\ldots,x_n]$ or $K[[x_1,\ldots,x_n]]$.
 One of the
main results in the theory of $D$-modules is that both the
localizations $S_f$ and the local cohomology modules $H^i_{I}(S)$
have finite length as $D_{S|K}$-modules. If $K$ is a field of
characteristic zero, this result follows from the fact that
localizations and local cohomology modules belong to the class of
{\it holonomic} $D_{S|K}$-modules (see \cite{Bj1,Cou}).
These results can be extended to differentiably admissible $K$-algebras, which is a larger 
class or regular rings  containing the rational numbers (see \cite{MeNa,NunezDM}).  
For intance, if $(R,\m,K)$ is a ramified regular ring of mixed characteristic $p>0$, 
then $R[1/p]$ is  a differentiably admissible algebra  \cite{NunezDM}.

In positive characteristic, these results were proved by Lyubeznik using his theory of {\it $F$-finite $F$-modules} \cite{LyuFmod}. 
In fact, he showed that $S_f$ and  $H^i_{I}(S)$ have finite length as $D_{S|K}$-modules
for a larger class of Noetherian regular $K$-algebras of prime characteristic.
Indeed, $F$-finite $F$-modules behave quite likely as holonomic modules in characteristic zero.

Despite the fact that we do not have a theory of holonomic or
$F$-finite $F$-modules for direct summands, we are still able to
prove the finite length of localization and local cohomology modules in this setting.

\begin{theorem}\label{ThmFinLenLC}
Let $K$ be a field and $S$ be either  $K[x_1,\ldots,x_n]$ or $K[[x_1,\ldots,x_n]]$.
Let $K\subseteq
R\subseteq S$ be a subring with a splitting $\beta:S\to R.$ Then
$R_f$ and $H^i_I(R)$ have finite length as $D_{R|K}$-modules for
every $f\in R$ and any ideal $I\subseteq R.$ Furthermore,
$\lambda_{D_{R|K}}(R_f)\leq \lambda_{D_{S|K}}(S_f)$ and
$\lambda_{D_{R|K}}(H^i_I(R))\leq \lambda_{D_{S|K}}(H^i_{IS}(S)).$
\end{theorem}

\begin{proof}
Since $R\subseteq S$ is a differential splitting, we have that
$R_f\subseteq S_f$ is a differential splitting for every $f\in R$ by
Proposition \ref{PropLoc}. Then we have $\lambda_{D_{R|K}}(R_f)\leq
\lambda_{D_{S|K}}(S_f)$ by Proposition \ref{PropFinLenght}. Let
$I=(f_1,\ldots,f_\ell)\subseteq R,$ and
$\underline{f}=f_1,\ldots,f_\ell.$ We note that
$\Cech^\bullet(\underline{f};R)\subseteq
\Cech^\bullet(\underline{f};S)$ is a complex of differential direct
summands, and so, the local cohomology, $H^i_I(R)\subseteq
H^i_{IS}(S)$, is again a differential direct summand by Lemma
\ref{LemmaAb}. Then, $\lambda_{D_{R|K}}(H^i_I(R))\leq
\lambda_{D_{S|K}}(H^i_{IS}(S))$ by  Proposition \ref{PropFinLenght}.
\end{proof}

We may extend the property of finite length  to other
$D_{R|K}$-modules if we consider the category $C(R,A)$ introduced by
Lyubeznik \cite{Lyu2,LyuUMC}.% to deal with the lack of holonomic
%$D$-modules in positive and mixed characteristic.

\begin{definition}
We denote by $C(R,A)$
the smallest subcategory of $D_{R|A}$-modules that
contains $R_f$ for all $f\in R$ and that is closed under subobjects,
extensions and quotients.
\end{definition}

The category $C(R,A)$ is also closed under localization and, in some
sense, behaves as well as holonomic or F-finite F-modules. In
particular, this category contains  iterated local cohomology
modules $H^{i_1}_{I_1}\cdots H^{i_\ell}_{I_\ell}(R)$, and more
generally, Lyubeznik functors $\mathcal{T}$ defined in
\cite{LyuDMod}.

 \begin{theorem}\label{ThmFinLenC(R,K)}
Let $K$ be a field and $S$ be either  $K[x_1,\ldots,x_n]$ or $K[[x_1,\ldots,x_n]]$.
Let $K\subseteq
R\subseteq S$ let be a subring with a splitting $\beta:S\to R.$
Then, every $D_{R|K}$-module in $C(R,K)$ has finite length as
$D_{R|K}$-module.
\end{theorem}
\begin{proof}
This statements follows from Theorem \ref{ThmFinLenLC}, since $R_f$
has finite length as $D_{R|K}$-module, and every object in $C(R,K)$
is build from localization modules with operations that preserve the
property of having finite length.
\end{proof}

From the previous theorem, we obtain as a corollary a result by  the
third author regarding associated primes of local cohomology.

\begin{corollary}[{\cite{NunezDS}}]
Let $K$ be a field and $S$ be either $K[x_1,\ldots,x_n]$ or $K[[x_1,\ldots,x_n]]$.
Let $K\subseteq R\subseteq S$ let be a subring with a splitting $\beta:S\to R.$
Then, $\Ass_RH^i_I(R)$ is a finite set for every ideal $I\subseteq R$ and every integer $i.$
\end{corollary}
\begin{proof}
We note that a simple $D_{R|K}$-module, $M\neq 0$, has only one associated
prime. In fact, if  $Q$ is a maximal element in the set of the
associated  primes of $M$, then $\Ass_R H^0_{Q}(M)=Q$  and
$H^0_{Q}(M)$ is a nonzero $D_{R|K}$-submodule of $M$. Since
$H^i_I(R)$ has finite length, there exists an increasing sequence of
$D_{R|K}$-modules $0=M_0\subseteq M_1\subseteq \ldots \subseteq
M_{\ell}=H^i_I(R)$ such that   $M_{j+1}/M_j$ is a simple
$D_{R|K}$-module. Then, $\Ass_RH^i_I(R)\subseteq \bigcup_{j=0}^\ell
\Ass_R M_{j+1}/M_j.$
\end{proof}

\begin{remark}
There are other structural properties that local cohomology modules
satisfy over regular rings containing a field, which are proved
using $D$-modules. For instance, local cohomology modules have finite
Bass numbers and their injective dimension is bounded above by the
dimension of their support. Both of these properties fail for direct
summands. For instance, $R=K[ac,ad,bc,bd]$
is a direct summand of
$S=K[a,b,c,d]$, $H^2_{(ab,ad)(R)}$ is a local cohomology module with
an infinite socle \cite{HartCD}, so not all Bass numbers are finite.
In addition,  $H^2_{(ab,ad)(R)}$ is a simple $D_{R|K}$-module
\cite{Hsiao} supported at the maximal ideal $(ab,ad,cd,bc)R,$ which
is not an injective $R$-module, so the injective dimension can be
bigger than the dimension of the support.
\end{remark}
%%%%%%%%%%%%%%%%%%%%%%%%%%%%%%%%%%%%%%%%%%%%%%%%%%%%%%%%%%%%%%
\subsection{Bernstein-Sato polynomial over direct summands}
\label{BS-poly}
%%%%%%%%%%%%%%%%%%%%%%%%%%%%%%%%%%%%%%%%%%%%%%%%%%%%%%%%%%%%%%

Let $S$ be either the polynomial ring $K[x_1,\ldots,x_n]$ or the
formal power series ring  $K[[x_1,\ldots,x_n]]$, where $K$ is a
field of characteristic zero and $ R\subseteq S$ a direct summand.
We aim to extend the well-known theory of Bernstein-Sato over $S$ to
the direct summand $R$. We start reviewing the basic facts.

Recall that $D_{S|K}=S\left\langle \frac{d}{d{x_1}}, \ldots,
\frac{d}{d{x_n}}\right\rangle$ is the ring of $K$-linear
differential operators of $S$. Let $D_{S|K}[s]$ be the
polynomial ring in the indeterminate $s$ and coefficients in
$D_{S|K}$. For any nonzero element $f\in S$, there exists a
polynomial $b(s)\neq 0$ and a differential operator $\delta(s)\in
D_{S|K}[s]$ satisfying the following functional equation:
\begin{equation}\label{FunctionalEquation}
\delta(s)\cdot f\boldsymbol{f^s}=b(s) \boldsymbol{f^s}.
\end{equation}
This is an identity in $M[\boldsymbol{f^s}]:=S_f[s]\boldsymbol{f^s}$
which is the rank one free $S_f[s]$-module with generator
$\boldsymbol{f^s}$ viewed as a formal symbol. After
taking the specializations $M[\boldsymbol{f^s}]\to S_f,$ given by
$\boldsymbol{f^s}\mapsto f^{t}$ for $t\in\ZZ,$ we obtain the
following equation
\begin{equation}\label{EqBS}
\delta(t)\cdot f^{t+1}=b(t) f^t.
\end{equation}

The set of all polynomials $b(s)$ for which there exists an operator
$\delta(s)$ such that Equation \ref{FunctionalEquation} holds forms
an ideal of $K[s]$. The monic polynomial that generates this ideal
is called the Bernstein-Sato polynomial \cite{BernsteinPoly,SatoPoly} and is
denoted as $b_f(s)$. 

A fundamental result  states that the roots of $b_f(s)$ are negative
rational numbers in  the polynomial case\footnote{  Indeed, this result is also  stated in the analytic case ($S=\mathbb{C}\{x_1,\dots,x_n\}$).}  \cite{Malgrange2,Kashiwara}. 
As a consequence,  we may view $b_f(s)\in \QQ[s]$.

A key point in this theory is that we can give a $D_{S|K}[s]$-module
structure to $M[\boldsymbol{f^s}]$ as follows: The action of the
derivatives on a general element $g \boldsymbol{f^s}$ is given by
$$
\frac{d}{d{x_i}} \cdot g \boldsymbol{f^s}= \left(
\frac{d}{d{x_i}}\cdot  g +
 sgf^{-1}\frac{d}{d{x_i}}(f)
\right) \boldsymbol{f^s},
$$ and the multiplication by $s$ is $D_{S|K}$-linear.
Unfortunately, when we turn our attention to a non-regular ring $R$,
we are not able to define a  $D_{R|K}[s]$-module structure on
$R_f[s]\boldsymbol{f^s}$. However, it still makes sense to study the
existence of an element $\delta(s)\in D_{R|K}[s]$ and a polynomial
$b(s)\in \QQ[s]$ that satisfy Equation \ref{EqBS} for every
$t\in\ZZ.$

In order to make precise the setup that  we
introduce the following general definition that encodes all the
Bernstein-Sato data.
{
\begin{definition}\label{Definition B-S}
Let $A\subseteq R$ be two Noetherian rings, and $f\in R$. We
consider the $A[s]$-module,
$$\mathcal{B}=\left\{(\delta(s),b(s)) \in D_{R|A}[s] \times A[s] \;|\; \delta(t)\cdot
f^{t+1}=b(t) f^t\; \forall t\in\ZZ\hbox{ in }R_f\right\}.$$ Let
$\mathcal{B}_f$ be the image of $\mathcal{B}$ under the projection
map on the second entry. If $\mathcal{B}_f=0$, we say that $f$ has
no Bernstein-Sato polynomial.  If $\mathcal{B}_f\neq 0,$ we call its
monic generator  the Bernstein-Sato polynomial of $f$, which is denoted by $b^{R_A}_f(s)$.
If the rings $R$ and $A$ are clear from the context, we just write
$b_f(s).$
\end{definition}
}

Our next result shows that  Bernstein-Sato polynomials exist  for a
direct summand, $R\subseteq S$, as long as they exist for $S$.

\begin{theorem}\label{ThmGralBShyp}
Let $A\subseteq R\subseteq S$  be three Noetherian rings. Suppose
that the inclusion $R\subseteq S$ has a splitting $\beta:S\to R$.
Let $f\in R.$ If  $b^{S_A}_f(s)$ exists, then $b^{R_A}_f(s)$ also
exists. Furthermore, $b^{R_A}_f(s)$ divides $b^{S_A}_f(s)$.
\end{theorem}
\begin{proof}
Let $\delta(s)\in D_{S|A}[s]$ be such that $\delta(t)\cdot
f^{t+1}=b(t)f^t$  for all $t\in\ZZ$, where the equality is taken in
$S_f$. Since $R_f$ is a differential direct summand of $S_f$, we
have that
\begin{equation}\label{EqBSDS}
\beta\circ \delta(t)_{|_{R}} \cdot f^{t+1}=b(t) f^t
\end{equation}
for every $t\in \ZZ$ in $R_f$. Then, $b^{S_A}_f(s)$ is a multiple of
$b^{R_A}_f(s)$, since this  is the monic polynomial satisfying \ref{EqBSDS}.
\end{proof}

\begin{corollary} \label{BSpoly_exists}
Let $K$ be a field and $S$ be either $K[x_1,\ldots,x_n]$ or $K[[x_1,\ldots,x_n]]$. Let $K\subseteq R\subseteq S$ be a direct summand with an
splitting $\beta:S\to R$. Then the Bernstein-Sato polynomial exists
for every element $f\in R.$ Furthermore, in the polynomial case all the roots of $b^{R_K}_f(s)$ are
negative rational numbers if $S$ is polynomial.
\end{corollary}

\begin{proof}
The Bernstein-Sato polynomial $b^{S_K}_f(s)$ exists \cite{BernsteinPoly,SatoPoly}. If $S$ is polynomial,  
$b^{S_K}_f(s)$  only has negative rational roots  \cite{Malgrange2,Kashiwara}.
\end{proof}

\begin{remark}
Theorems \ref{ThmFinLenLC}, \ref{ThmFinLenC(R,K)}, and \ref{ThmGralBShyp} can be extended to direct 
summands of  differentiably admissible $K$-algebras, where $K$ is a field of characteristic zero 
(see \cite{MeNa,NunezDM}). In particular, we may also consider direct summands of $S=\mathbb{C}\{ x_1,\dots,x_n\}$,
and the rationality of the roots of the Bernstein-Sato polynomials also holds  \cite{Malgrange2,Kashiwara}.
\end{remark}

We now present an example that show that $b^R_f(s)$ could differ from $b^S_f(s).$ This example was kindly given to us by Jack Jeffries.

\begin{example}\label{ExampleDiffBS}
Let $R=\mathbb{Q}[xu,yv]$ and $S=\mathbb{Q}[x,y,u,v]$
We note that $R=S^G$, where $G$ is a torus with coordinates $\lambda_1, \lambda_2$ acting by 
$$
x\mapsto \lambda_1 \lambda_2 x,\; y\mapsto \lambda_1 \lambda_2^{-1} y,\; u\mapsto \lambda_1^{-1} \lambda_2^{-1} u, \; v\mapsto \lambda_1^{-1} \lambda_2 v.
$$
Consider $f=xu-yv$. 
There exists an isomorphism from $R$  
to a polynomial ring in two variables, which sends $f$ to a variable. Thus,
 $b^R_f(s)=s+1$ and $R_f$ is generated by $\frac{1}{f}$ as a $D_{R|K}$-module.
In contrast, $b^S_f(s)=(s+1)(s+2)$, and so, $S_f\neq D_{S|K}\cdot \frac{1}{f}$ \cite[Lemma 1.3]{WaltherBS}. 
Furthermore, $H^i_f(R)$ is a simple 
$D_{R|K}$-module, but $H_{f}^1(S)$ is not a simple $D_{S|K}$-module.
\end{example}

\begin{remark}\label{RmkGralBS}
There are several variants and extensions where the theory of
Bernstein-Sato polynomial applies. The results of Theorem
\ref{ThmGralBShyp} and Corollary \ref{BSpoly_exists} can be extended
to these cases:

\vskip 2mm

$\bullet$ {\it Bernstein-Sato polynomial associated to an algebraic
variety:} Let $I\subseteq S$ be an ideal generated by a sequence of
nonzero polynomials $\underline{f} =f_1,\ldots, f_\ell.$ Budur,
Musta\c{t}\u{a} and Saito \cite{BMSBS} defined the Bernstein-Sato polynomial,
$b_I(s)$, associated to the ideal $I$. After a convenient
shifting, $b_I(s)$ is also an invariant of the variety $V(I)$. They proved
the existence of a nonzero polynomial $b(s)\in\QQ[s]$, and
differential operators $\delta_c (s_1,\ldots,s_\ell) \in
D_{R|K}[s_1,\ldots, s_\ell]$ satisfying, in our setting, a
functional equation
\begin{equation}\label{FunctionalEquationIdeals}
\sum_{c \in \ZZ^\ell} \; \prod_{\substack{i=1,\ldots, \ell\\ c_i<0}}
\delta_c (t_1,\ldots,t_\ell) \cdot \binom{t_i}{c_i}
\prod^{\ell}_{i=1} f^{t_i+c_i}_i=b(t_1+\ldots
+t_r)\prod^{\ell}_{i=1} f^{t_i}_i,
\end{equation}  for every $t_i\in\ZZ$ and
 where the sum varies over finitely many  $c=(c_1,\dots, c_\ell)\in \ZZ^\ell$.
We think of this as a functional equation in $S_{f_1\cdots f_\ell}$.

Let $R\subseteq S$ be a direct summand with a splitting $\beta:S\to
R$ and assume that $f_i\in R$, $i=1,\dots,\ell$. Then  we obtain the
analogue of Equation \ref{FunctionalEquationIdeals} for $R$ by
composing the differential operators $\delta_c (t_1,\ldots,t_\ell)$
with the splitting $\beta$.

\vskip 2mm

$\bullet$ {\it Bernstein-Sato ideal associated to a family of
hypersurfaces:} Let  $\underline{f} =f_1,\ldots, f_\ell $ be a
sequence of elements in $S$. Then, there exists a nonzero polynomial
$b(s_1,\dots, s_\ell)\in\QQ[s_1,\dots, s_\ell]$, and differential
operators $\delta(s_1,\ldots,s_\ell) \in D_{R|K}[s_1,\dots, s_\ell]$
satisfying, in our setting, a functional equation
\begin{equation}\label{EqBS_sequence}
\delta(t_1,\ldots,t_\ell)\cdot \prod^{\ell}_{i=1} f^{t_i+1}_i=
b(t_1,\ldots,t_\ell) \prod^{\ell}_{i=1} f^{t_i}_i.
\end{equation}
for every $t_i\in\ZZ$. The ideal in $\QQ[s_1,\dots, s_\ell]$
generated by these polynomials $b(s_1,\dots, s_\ell)$ is the
Bernstein-Sato ideal associated to the sequence $\underline{f}$. If $R\subseteq
S$ is a direct summand and $f_i\in R$, $i=1,\dots,\ell$, we use the same compose by $\beta$ again to obtain the  equation in $R$.

\vskip 2mm

% $\bullet$ {\it Bernstein-Sato polynomials for modules:}
We end this subsection with a example, generously suggested by Jack Jeffries, which shows that Bernstein-Sato polynomials do not exist in general.
  
\begin{example}\label{ExampleNoBS}
Let $R=\CC[x,y,z]/(x^3+y^3+z^3)$.
We consider $R$ as a graded ring where every variable has degree $1$, and we denote the $\CC$-subspace of  homogeneous elements of degree $\ell$ by $[R]_\ell$.
We note that $D_{R|\CC}$ is also a graded ring, where a differential operator has degree $\ell\in\ZZ$ if $\delta  g=0$ or $\deg(\delta  g)=\deg(g)+\ell $ for every homogeneous element  $g\in R$. 
We note that this identity extends to any localization of $R$ at a homogeneous element.
The $\CC$-subspace of differential operators of degree $\ell$ is denoted by $D_{R|\CC}(\ell)$.
In this case, $D_{R|\CC}=\bigoplus_{\ell\in \ZZ} D_{R|\CC}(\ell)$. Furthermore, $D_{R|\CC}(\ell)=0$ for every $\ell<0$ \cite[Proposition 1]{DiffNonNoeth}. 
Let $f\in R$ be any element of positive degree $u$.
We now show by contradiction that $f$ has no Bernstein-Sato polynomial (as in Definition \ref{Definition B-S}).
Suppose that there exists $\delta[s]\in D_{R|\CC}[s] $ and $b(s)\in\CC[s]$ such that  $b(s)\neq 0$ and  $\delta(t)f^{t+1}=b(t)f^t$ for every $t\in\ZZ$.
We fix $t\in\ZZ$ such that $t>0$ and $b(t)\neq 0.$
We set $\delta =\delta(t),$ and we take $\delta_\ell\in D_{R|\CC}(\ell)$ such that
$\delta=\delta_0+\cdots+ \delta_{k}$, which is possible because there are no operators of negative degree \cite[Proposition 1]{DiffNonNoeth}.  Then,
$$
b(t)f^{t}=\delta f^{t+1}=\delta_0 f^{t+1}+\cdots+ \delta_{k} f^{t+1}\in \bigoplus_{\ell >(t+1)u} [R]_\ell .
$$
As a consequence,  $tu>(t+1)u$ for every $\ell$ appearing in the equation above, because 
$b(t)f^t$ is a nonzero homogeneous element of degree $tu$. This is a contradiction because $u>0.$
\end{example}

\begin{remark}
We note that even in the case where $D^1_{R|K}$ has no elements of negative degree, one can find elements of negative degree in $D_{R|K}$. For instance, if $R=\CC[a,b,c,d]/(ad-bc)\cong \CC[xy,xz,yu,zu],$ which is a direct summand of $S=\CC[x,y,z,u]$ with splitting $\beta:S\to R$ defined on the monomials by $\beta(x^{\alpha_1} y^{\alpha_2} z^{\alpha_3}u^{\alpha_4})=x^{\alpha_1} y^{\alpha_2} z^{\alpha_3}u^{\alpha_4}$ if $\alpha_1+\alpha_4=\alpha_2+\alpha_3$ and 
$\beta(x^{\alpha_1} y^{\alpha_2} z^{\alpha_3}u^{\alpha_4})=0$ otherwise.
Since $R$ is a graded ring with an isolated singularity, there are no elements of order $1$ and negative degree \cite{Kantor}. However, $\beta\circ \frac{\partial}{\partial x}\frac{\partial}{\partial y}$ is a differential operator on $R$ of order $2$ and negative degree.
\end{remark}

Theorem \ref{ThmGralBShyp} and Example \ref{ExampleNoBS} motivate the following question.

\begin{question}
What conditions are necessary and sufficient for a $K$-algebra for the existence of the Bernstein-Sato polynomial?
\end{question}

\end{remark}

%%%%%%%%%%%%%%%%%%%%%%%%%%%%%%%%%%%%%%%%%%%%%%%%%%%%%%%%%%%%%
\section{$F$-invariants of direct summands}

%%%%%%%%%%%%%%%%%%%%%%%%%%%%%%%%%%%%%%%%%%%%%%%%%%%%%%%%%%%%%%
In this section we study different invariants in positive
characteristic: $F$-jumping numbers and $F$-thresholds.
We point out that in a regular ring these two families of invariants are the same. However, over singular rings they
usually differ, even for direct summands. For this reason, we need to study them using different approaches.

We recall that  $R$ acquires an structure of $R$-module by restriction of scalars via the $e$-th iteration of the Frobenius map. $F^e$. We denote this module action on $R$ by $F^e_* R$. Namely, we have $r \cdot F^e_*x = F^e_*(r^{p^e} x) \in F^e_* R$ for $r \in R$ and $F^e_*x \in F^e_* R$.
%Since $R$ is $F$-finite, $F^e_* R$ is a finitely generated $R$-modules. 
If $R$ is a reduced ring, it is a common practice to identify $F^e_* R$ with $R^{1/p^e},$ via the isomorphism of $R$-modules $F^e_*r\mapsto r^{1/p^e}.$

Throughout this section we consider rings
containing a field of positive characteristic that are $F$-finite,
that is, rings such that $F^e_* R$ is a finitely generated $R$-module.

%%%%%%%%%%%%%%%%%%%%%%%%%%%%%%%%%%%%%%%%%%%%%%%%%%%%%%%%%%%%%
\subsection{Test ideals  of direct summands}
%\subsection{Cartier operators and differential operators}
%%%%%%%%%%%%%%%%%%%%%%%%%%%%%%%%%%%%%%%%%%%%%%%%%%%%%%%%%%%%%%

In this subsection we prove that the $F$-jumping numbers of an ideal
$I$ of a direct summand $R$ of a regular ring $S$ are contained in
the $F$-jumping numbers of the extended ideal $I S.$ This result
follows from comparison of rings of differential operators and
Cartier morphisms between $R$ and $S$.  Differential operators have
been used previously to show discreteness and rationality of
$F$-jumping numbers for regular rings. However, they have not been
used in the singular case as far as we know.

We point out that  Chiecchio, Enescu, Miller, and Schwede
\cite{CEMS} used the fact that $R$ is a direct summand of the symbolic
Rees  algebra $\cR=\oplus^{\infty}_{t=0} \cO_X (-t K_X ) $, to prove discreteness and rationality of
$F$-jumping numbers in certain cases, given that $\cR$ is finitely
generated as $R$-algebra. From recent results in the minimal model program  \cite{MMP1,MMP2,HC}
(see also \cite[Theorem 92]{Kol10}), one could expect that $\cR$ is finitely
generated for every strongly $F$-regular ring, in particular, for
direct summands of regular rings. However, to the best of our
knowledge, this question is still open.
Furthermore, this question relates to the longstanding problem of weakly $F$-regular versus strongly $F$-regular  \cite{CEMS}.
We point out that our strategy differs from theirs by using
differential operators over a singular ring.

\begin{definition}
Let $R$ be an $F$-finite Noetherian ring.

$\bullet$ We say that $R$ if $F$-pure if the inclusion map $R\to F^e_* R$ splits for all $e>0$.

$\bullet$ Assume that $R$ is a domain. We say that $R$ is strongly $F$-regular if for every $r\in R$
there exists $e\in\NN$ such that the $\varphi:R\to F^e_* R$
defined by $1\mapsto F^e_* r$ splits.
\end{definition}

Regular rings are strongly $F$-regular. Moreover, direct summands of
strongly $F$-regular rings are strongly $F$-regular by a result of \cite{HHStrongFreg}.

We now introduce the definition and preliminary results regarding
Cartier operators and its relation with differential operators in
positive characteristic.

\begin{definition} \label{FrobeniusCartier}
Let $R$ be an $F$-finite ring.

$\bullet$ An additive map $\psi:R\to R$ is a $p^{e}$-linear map if
$\psi(r f)=r^{p^e}\psi(f).$ Let $\mathcal{F}^e_R$ be the set of all
the $p^{e}$-linear maps. Then, we have  $\mathcal{F}^e_R \cong
\Hom_R(R, F^e_* R)$.

\vskip 2mm

$\bullet$ An additive map $\phi:R\to R$ is a $p^{-e}$-linear map if
$\phi(r^{p^e} f)=r\phi(f).$ Let $\cC^e_R$ be the set of all the
$p^{-e}$-linear maps. Then, we have  $\cC^e_R \cong
\Hom_R(F^e_* R,R)$.

\end{definition}

%\begin{remark}\label{RemCorresCartier}
%There is a bijective correspondence between $\Psi:\cC^e_R
%\to\Hom_R(F^e_* R,R)$  given by $\Psi(\phi)(r^{1/p^e})=\phi(r)$
%for $\phi\in\cC^e_R$.
%\end{remark}

%\begin{definition}
%Let $R$ be an $F$-finite Noetherian reduced ring.
%We say that $R$ if $F$-pure if the inclusion map $R\to F_* R $ splits.
%\end{definition}

%\begin{definition}
%Let $R$ be an $F$-finite Noetherian domain.
%We say that $R$ is strongly $F$-regular if for every $r\in R$
%there exists $e\in\NN$ such that the $\varphi:R\to F^e_* R$ defined by $1\mapsto r^{1/p^e}$ splits.
%\end{definition}

The main examples of $p^e$ and $p^{-e}$-linear maps are the iterated
Frobenius and Cartier morphisms. For this reason, these morphisms
are usually referred as Frobenius and Cartier operators. They have a
close relation with differential operators. Recall from Section
\ref{Dmod} that for $F$-finite rings we have $$D_R= \bigcup_e
D^{(e)}_R.
$$
Since $R$ is reduced, we have $D^{(e)}_R\cong
\Hom_R(F^e_* R,F^e_* R)$. It is natural to consider the pairing
$ \cC^e_R \otimes_R \mathcal{F}^e_R \longrightarrow D^{(e)}_R$
sending $\phi \otimes \psi$ to its composition $\psi \circ \phi$. If
$R$ is a regular ring, the previous map is an isomorphism of Abelian
groups, that is $ \cC^e_R \otimes_R \mathcal{F}^e_R \cong
D^{(e)}_R.$ One concludes, roughly speaking, that it is equivalent
to consider differential operators or just Cartier operators.

In general, for singular rings, we do not have such an isomorphism.
However, in the $F$-pure case we make an easy observation that relates differential operators
with Cartier maps. This relation is the heart of Lemma
\ref{LemmaRestrictionDmodCartier},  which is a key part of our
strategy towards Theorem \ref{MainFjump}.

\begin{remark}\label{Rem Onto}
Let $R$ be an $F$-pure $F$-finite ring.
If $\gamma:F^e_* R\to R$ is a splitting, the map
$\Psi:\Hom_R(F^e_* R,F^e_* R)\to \Hom_R(F^e_* R,R)$ defined  by $\delta\mapsto \gamma\circ\delta$ is surjective.
Then, the corresponding map
$D^{(e)}_R \to  \cC^e_R$  is also surjective.

\end{remark}

Let $S$ be a regular ring.  The so-called {\it Frobenius descent},
that follows from the isomorphism $\cC^e_S \otimes_S \mathcal{F}^e_S
\cong D^{(e)}_S$ and Morita equivalence, is a key ingredient  to
show that $S_f=D_S \cdot \frac{1}{f}$ \cite[Theorem 1.1]{AMBL}. This
is an important result regarding the $D$-module structure of the
localization $S_f$ at an element $f\in S$.  This result was extended
by Takagi and Takahashi \cite[Corollary 2.10]{TT} to the case of
graded rings with {\it finite F-representation type}.  
This class of rings include graded direct summands.
In our first
result in this section, we prove that this occurs for all  direct
summands of $F$-finite regular domains.

\begin{theorem}\label{PropLocCyclicPrime}
Let $S$ be a regular $F$-finite domain. Let $R\subseteq S$ be an
extension of Noetherian rings such that $R$ a direct summand of $S$.
Then $R_f$ is generated by $\frac{1}{f}$ as $D_R$-module.
\end{theorem}
\begin{proof}
Since $S_f$ is generated as $D_S$-module by $\frac{1}{f}$, for
every $e$ there exists $\delta\in D^{(e')}_S$, with $e'\geq e$, such
that $\delta (\frac{1}{f})=\frac{1}{f^{p^e}}.$ Then,  by multiplying
this equation by $f^{p^{e'}},$ we obtain
$\delta(f^{p^{e'}-1})=f^{p^{e'}-p^{e}}.$

Let  $\beta:S\to R$ denote a splitting. Let $\widetilde{\delta}=\beta\circ\delta.$
We note that $\widetilde{\delta}\in D^{(e')}_R$. In addition,
$\widetilde{\delta}(f^{p^{e'}-1})=f^{p^{e'}-p^{e}}.$
Hence, $\widetilde{\delta}(\frac{1}{f})=\frac{1}{f^{p^e}}.$ We conclude that
$R_f$ is generated as $D_R$-module by $\frac{1}{f}.$
\end{proof}

%%%%%%%%%%%%%%%%%%%%%%%%%%%%%%%%%%%%%%%%%%%%%%%%%%%%%%%%%%%%%%%%%%%
%\subsection{Test ideals  of direct summands}
%%%%%%%%%%%%%%%%%%%%%%%%%%%%%%%%%%%%%%%%%%%%%%%%%%%%%%%%%%%%%%%%%%%

Test ideals have been a fundamental tool in the theory of tight closure
developed by Hochster and the second author \cite{HoHu1,HoHu2,HoHu2}.
Hara and Yoshida \cite{H-Y} extended the theory to include
test ideals $\tau_R(I^\lambda)$ associated to pairs $(R,I^\lambda)$ where $I\subseteq R$ is an
ideal and $\lambda \in \mathbb{R}$ is a parameter.
A new approach to test ideals\footnote{Indeed, they describe the generalization to pairs of the big test
ideal considered in \cite{HoHu2}.} by means of Cartier operators was given by
Blickle, Musta\c{t}\u{a} and Smith  \cite{BMS-MMJ,BMS-Hyp}  in the case that
$R$ is a regular ring. Their approach has been extremely useful in order to extended
the theory of test ideals to non-regular rings. We refer
to  \cite{TestQGor,BB-CartierMod,BlickleP-1maps}
for a more  general setting.

\begin{definition}
Suppose that $R$ is a strongly $F$-regular ring. Let $I\subseteq R$
be an ideal, and $\lambda\in \RR_{>0}.$ The  test
ideal\footnote{Precisely, this are the big test ideals. This ideals
are equal to the original test ideals for graded rings \cite{FregEquiv} and
Gorenstein rings \cite{LyuKaren}.} of  the pair $(I,\lambda)$ is
defined by
$$
\tau_R(I^{\lambda})=\bigcup_{e\in\NN}\cC^e_R I^{\lceil p^e \lambda \rceil}.
$$
\end{definition}

We note that the chain of ideals $\{ \cC^e_R I^{\lceil p^e \lambda
\rceil}\}$   is increasing \cite[Proposition 4.4]{TT}, and so,
$\tau_R(I^{\lambda})=\cC^e_R I^{\lceil p^e \lambda \rceil}$ for
$e\gg 0$.

\vskip 2mm

We now summarize basic well-known properties of test ideals.
In the case of strongly $F$-regular rings we refer to
\cite[Lemma 4.5]{TT}. For the general case of $F$-finite rings, we refer \cite[Proposition ]{BlickleP-1maps}.

\begin{proposition}\label{PropBasics}
Let $R$ be an $F$-finite ring, $I,\b\subseteq R$ ideals, and $\lambda,\lambda'\in\RR_{>0}.$
Then,
\begin{enumerate}
\item If $I\subseteq \b,$ then $\tau_R(I^\lambda)\subseteq \tau_R(\b^{\lambda})$.
\item If $\lambda<\lambda',$ then $\tau_R(I^{\lambda'})\subseteq \tau_R(I^{\lambda})$.
\item There exists $\epsilon>0,$ such that $\tau_R(I^{\lambda})= \tau_R(I^{\lambda'})$,
if $\lambda'\in [\lambda,\lambda+\epsilon)$.
\end{enumerate}
\end{proposition}

Therefore, every ideal $I\subseteq R$ is  associated to a family of test ideals $\tau_R(I^{\lambda})$
parameterized by real numbers $\lambda \in \mathbb{R}_{>0}$. Indeed, they form a nested chain of ideals.
%$$ R \varsupsetneq \tau_R(I^{\lambda_1}) \varsupsetneq  \tau_R(I^{\lambda_2}) \varsupsetneq \cdots $$
%and
The real numbers where the test ideals change are called {\it $F$-jumping numbers}. We now make this precise,

\begin{definition}
Let $R$ be an $F$-finite ring and let $I\subseteq R$ be an ideal. A real number $\lambda$ is an $F$-jumping number
of $I$ if
$$
\tau_R(I^{\lambda})\neq \tau_R(I^{\lambda-\epsilon})
$$
for every $\epsilon>0.$
\end{definition}

\begin{remark}
Let $R$ be a strongly $F$-regular $F$-finite ring of characteristic
$p > 0,$ and  $I\subseteq R$. If $R$ is either a local ring or a
graded $K$ algebra with $I$ homogeneous,  then the $F$-pure threshold of $I$ is defined in terms of maps $R\to F^e_* R$ that split \cite[Definition 2.1]{TW2004}. Specifically,
$$
\fpt(I):=\sup\{\lambda\;|\; \forall e\gg 0\; \exists f\in I^{\lfloor (p^e-1)\lambda\rfloor}\hbox{ such that }1\mapsto F^e_* f\hbox{ splits}\}.
$$
Then, $\fpt(I)$
is the first $F$-jumping number of $I.$
\end{remark}

From now on we consider extensions of Noetherian rings  $R\subseteq S$ such that $R$ is a direct
summand of $S$. Our aim is to relate test ideals of an ideal in $R$ with the test ideal of
its extension to $S$. The first result in this direction  shows that
the $F$-pure threshold of an ideal of $R$ is greater or
equal  than the $F$-pure threshold of its extension in $S$.

\begin{proposition}
Let $S$ be a regular $F$-finite domain. Let $R\subseteq S$ be an
extension of Noetherian rings such that $R$ is a direct summand of
$S$. Then, $\tau_S ((I S)^\lambda)\cap R\subseteq \tau_R
(I^{\lambda})$ for every ideal $I\subseteq R$ and
$\lambda\in\RR_{>0}.$ In particular, $\fpt_S(I S)\leq \fpt_R(I).$
\end{proposition}
\begin{proof}
Let $f\in \tau_S ((I S)^\lambda)\cap R.$ Then, there exists
$e\in\NN$ such  that $f\in \cC^e_S (I S)^{\lceil p^e \lambda
\rceil}.$ Then, there exist $g_1,\ldots,g_\ell\in I^{\lceil p^e
\lambda \rceil}$ and $\phi_1,\ldots,\phi_\ell\in C^S_e$ such that
$\phi_1(g_1)+\ldots+\phi_\ell (g_\ell)=f$.

Let  $\beta:S\to R$ denote a splitting.
Then, $\beta\circ {\phi_i}_{|_R}\in \cC^e_R.$
Thus, $\beta (\phi_i(g_i))\in \cC^e_R I^{\lceil p^e \lambda \rceil}\subseteq \tau_R (I^{\lambda})$.
Then,
\begin{align*}
f&=\phi_1(g_1)+\ldots+\phi_\ell (g_\ell);\\
&=\beta(\phi_1(g_1)+\ldots+\phi_\ell (g_\ell));\\
&=\beta(\phi_1 (g_1))+\ldots + \beta(\phi_\ell (g_\ell)).
\end{align*}
Hence, $f\in \cC^e_R (I^{\lceil p^e \lambda \rceil})\subseteq \tau_R (I^{\lambda}).$
\end{proof}

The following results are  a key part of our
strategy towards Theorem \ref{MainFjump}.

\begin{lemma}\label{LemmaRestrictionDmodCartier}
Let $R$ be an $F$-pure $F$-finite ring, and $I,\b\subseteq R$ ideals.
If $D^{(e)}_R I=D^{(e)}_R\b$, then $\cC^e_R I=\cC^e_R \b$.
\end{lemma}
\begin{proof}
We have that
\begin{align*}
\cC^e_R I & =\Hom_R(F^e_* R,R)\cdot F^e_* I\hbox{ by Definition \ref{FrobeniusCartier}}\\
& =\gamma(\Hom_R(F^e_* R,F^e_* R)\cdot F^e_* I)\hbox{ by Remark \ref{Rem Onto}}\\
& =\gamma(\Hom_R(F^e_* R,F^e_* R)\cdot F^e_* \b)\hbox{ by Definition \ref{FrobeniusCartier} and }D^{(e)}_R I=D^{(e)}_R\b;\\
& =\Hom_R(F^e_* R,R)\cdot F^e_* \b\hbox{ by Remark \ref{Rem Onto}}\\
& =\cC^e_R \b\hbox{ by Definition \ref{FrobeniusCartier}}
\end{align*}
\end{proof}

\begin{lemma}\label{LemmaEqCartier}

Let $R\subseteq S$ be an extension of Noetherian rings such that $R$
is a direct summand of $S$. Let $I\subseteq R$ denote an ideal. If
$\cC^e_S (I S)^r =\cC^e_S (I S)^t$ for $r,t\in\NN,$ then $\cC^e_R
I^r=\cC^e_R I^t$.
\end{lemma}
\begin{proof}
We assume without loss of generality that $r\leq t.$
We have $\cC^e_R I^t\subseteq \cC^e_R I^r$, thus we focus on the other containment.
Since $\cC^e_S (I S)^r =\cC^e_S (I S)^t$, we have that
$$D^{(e)}_S (I S)^r =\left(\cC^e_S (I S)^r\right)^{[p^e]} = \left(\cC^e_S (I S)^t\right)^{[p^e]}=D^{(e)}_S (I S)^t.$$
Then,
$I^r\subseteq D^{(e)}_S (I S)^t.$
Let $f\in I^r.$
Then, there exists $\phi_1,\ldots,\phi_\ell\in D^{(e)}_S$ and $g_1,\ldots,g_\ell\in I^t$ such that
$f=\phi_1 (g_1)+\ldots +\phi_\ell(g_\ell).$
Let $\beta:S\to R$ denote a splitting.
Then, $\beta\circ {\phi_i}_{|_R}\in D^{(e)}_R $
and $\beta (\phi_i (g_i))\in D^{(e)}_R I^t$.
It follows that $f\in D^{(e)}_R (I^t)$ since
\begin{align*}
f&=\phi_1(g_1)+\ldots+\phi_\ell (g_\ell)\\
&=\beta(\phi_1(g_1)+\ldots+\phi_\ell (g_\ell))\\
&=\beta(\phi_1 (g_1)+\ldots + \beta(\phi_\ell (g_\ell)\\
&=\beta(\phi_1 (g_1))+\ldots + \beta(\phi_\ell (g_\ell)).
\end{align*}
 Therefore, $I^r\subseteq
D^{(e)}_R I^t.$ Hence, $D^{(e)}_R I^r\subseteq D^{(e)}_R \left(
D^{(e)}_R I^t\right)=D^{(e)}_R I^t.$ We conclude that $D^{(e)}_R
I^r=D^{(e)}_R I^t$, and so,  $\cC^e_R I^r=\cC^e_R I^t$ by Lemma
\ref{LemmaRestrictionDmodCartier}.
\end{proof}

We are now ready to prove the main result in this section.

\begin{theorem}\label{ThmEqTestIdeals}
Let $S$ be a regular $F$-finite domain. Let $R\subseteq S$ be an
extension of Noetherian rings such that $R$ is a direct summand of
$S$. Let $I\subseteq R$ denote an ideal, and
$\lambda_1,\lambda_2\in\RR_{>0}$. If $\tau_S((I
S)^{\lambda_1})=\tau_S((I S)^{\lambda_2}),$ then  $\tau_R
(I^{\lambda_1})=\tau_R (I^{\lambda_2})$.
\end{theorem}
\begin{proof}
There exists $t\in\NN$ such that
for $e>N,$
$\tau_S((I S)^{\lambda_1})=\cC^e_S (I S)^{\lceil p^e \lambda_1 \rceil}$ and
$\tau_S((I S)^{\lambda_2})=\cC^e_S (I S)^{\lceil p^e \lambda_2 \rceil}.$
Then, $\cC^e_S (I S)^{\lceil p^e \lambda_1 \rceil}=\cC^e_S (I S)^{\lceil p^e \lambda_2 \rceil}$ for $e\geq t.$
By Lemma \ref{LemmaEqCartier}, we have that
$\cC^e_R I^{\lceil p^e \lambda_1 \rceil}=\cC^e_R I^{\lceil p^e \lambda_2 \rceil}$ for $e\geq t.$
Then,
$$
\tau_R(I^{\lambda_1})
=\bigcup_{e\in\NN} \cC^e_R I^{\lceil p^e \lambda_1 \rceil}
=\bigcup_{e\geq t} \cC^e_R I^{\lceil p^e \lambda_1 \rceil}
=\bigcup_{e\geq t} \cC^e_R I^{\lceil p^e \lambda_2 \rceil}
=\bigcup_{e\in\NN} \cC^e_R I^{\lceil p^e \lambda_2 \rceil}
=\tau_R (I^{\lambda_2}).$$
\end{proof}

\begin{corollary}\label{CorMainDS}
Let $S$ be a regular $F$-finite domain. Let $R\subseteq S$ be an
extension of $F$-finite Noetherian rings such that $R$ is a direct
summand of $S$. Let $I\subseteq R$ denote an ideal. Then, the set
of $F$-jumping numbers of $I$ in $R$ is a subset of the set
of $F$-jumping numbers of $I S$ in $S$. In particular, the set
of $F$-jumping numbers of $I$ in $R$ is formed by rational numbers and
has no accumulation points.
\end{corollary}
\begin{proof}
The first claim follows immediately from the definition of
$F$-jumping  numbers and Theorem \ref{ThmEqTestIdeals}. The second
claim follows from the fact that the set of $F$-jumping numbers consits of rational numbers and it has no accumulation points for $F$-finite regular domains \cite[Theorem B]{ST-NonPrincipal}.
\end{proof}

\begin{remark}\label{RemAlg}
The previous result gives candidates to compute the $F$-jumping
numbers  of $I$. Furthermore, we obtain an algorithm to find the
test ideals and $F$-jumping numbers of $I$ given that we know the
corresponding objects for $I S.$

If $\tau_S(I S^\lambda)=\cC^e_S (I S)^{\lceil p^e \lambda \rceil}$
for  some $e\gg 0,$ then $\tau_R(I^\lambda)=\cC^e_R I^{\lceil p^e
\lambda \rceil}$. As a consequence, we can know whether a
$F$-jumping number of $I S$ is also a $F$-jumping number of $I$. To do this, we need to compute
first the next $F$-jumping number for $I S$, say $\lambda'$. We now
compute where the test ideal of $\tau_S(I S^\lambda)$ and
$\tau_S(I S^{\lambda'})$ stabilize. We use this to compute the test
ideals $\tau_R(I^\lambda)$ and $\tau_R(I^{\lambda'})$. Then, these
ideals differ if and only if $\lambda$ is an $F$-jumping number for
$I$ in $R$.
\end{remark}

It is natural to ask if the conclusion of Corollary \ref{CorMainDS}
holds in characteristic zero.

\begin{question}
Let $S=K[x_1,\ldots,x_n]$ be a polynomial ring over a field of characteristic zero.
Let $R\subseteq S$ be such that $R$
a direct summand of $S$. Let $I\subseteq R$ denote an ideal.
Is the  set of jumping numbers for the multiplier ideals of $I$ in $R$ a subset of the
set jumping numbers for the multiplier ideals of $I S$ in $S$?
\end{question}

%%%%%%%%%%%%%%%%%%%%%%%%%%%%%%%%%%%%%%%%%%%%%%%%%%%%%%%%%%%%%%%%%%%
\subsection{$F$-thresholds and Bernstein-Sato polynomials}
%%%%%%%%%%%%%%%%%%%%%%%%%%%%%%%%%%%%%%%%%%%%%%%%%%%%%%%%%%%%%%%%%%%

Musta\c{t}\u{a}, Takagi and Watanabe introduced 
some asymptotic invariants associated to an ideal $\a$
in an $F$-finite regular local ring $R$ \cite{MTW}. 
Roughly speaking, given any ideal $J\subseteq R$,
the $F$-threshold $c^J(\a)$ measures the containment of powers of $\a$ in the
Frobenius powers of $J$.
The set of $F$-thresholds when we vary $J$ coincides with the set of $F$-jumping
numbers of $\a$  \cite[Corollary 2.30]{BMS-MMJ}.
Together with the second author, they extended the notion of
$F$-threshold to all Noetherian rings in prime characteristic \cite{HMTW}. This definition is given by a limit, which they showed to exists for $F$-pure rings. This limit was recently showed to exist in full generality \cite[Theorem A]{DSNBP}.
In this subsection we extend previous relations between
$F$-threshold,  $F$-jumping numbers, and the Bernstein-Sato
polynomial, previously known only for polynomial rings \cite{MTW}.
For the sake of exposition, we restrict ourselves to principal
ideals and to finitely generated $\QQ$-algebras. We start by
recalling  this relation in the regular case. Let
$S=\QQ[x_1,\ldots,x_n]$. Let $\overline{f}$ denote the class of $f$
in $\FF_p[x_1,\ldots,x_n]$, where  $p$ is greater  any denominator
appearing in $f$. Then, for $p\gg 0,$
$$
b^S_f(\nu^{\a}_f(p^e))\equiv 0\quad \hbox{mod }p
$$
for every ideal $\a\subseteq \FF_p[x_1,\ldots,x_n]$ such that $\overline{f}\in\sqrt{\a}$ \cite[Proposition 3.11]{MTW}.

We start by recalling the definition of $F$-thresholds.

\begin{definition}
Let $R$ be a Noetherian rings, and $J,\a\subseteq R$ be ideals such
that $J\subseteq \sqrt{\a}.$ Let $\nu^\a_J(p^e)=\max\{t\;|\;
J^t\not\subseteq \a^{[p^e]}\}.$ We denote the $F$-threshold of $J$
with respect to $\a$ by
$$
c^\a(J)=\lim\limits_{e\to\infty} \frac{\nu^\a_J(p^e)}{p^e}.
$$
\end{definition}

It is worth mentioning that, in the singular case, the $F$-thresholds may differ
from the $F$-jumping numbers  (see \cite[Section 6]{HWY} for a few examples).
For this reason, it is not clear, a priori, that $F$-thresholds are rational  numbers that form a discrete set.
We note that  this is  the case if $R$ is a direct summand of a regular ring $S$.

\begin{proposition}
Let $S$ be a regular $F$-finite domain.
Let $R\subseteq S$ be a direct summand of $S$.
Then,
$c^\a (J)$
is a rational number for every ideal $J\subseteq R.$
Furthermore, the set
$$
\{c^\a (J) \; |\; \a\subseteq R \;\& \; J\subseteq \sqrt{\a} \}
$$
has no accumulation points.
\end{proposition}
\begin{proof}
We have that $c^\a (J)=c^{\a S}(JS)$  \cite[Proposition 2.2(v)]{HMTW}. Every $F$-threshold, $c^{\a S} (JS),$ is an $F$-jumping
number for test ideals of $\a$ \cite{BMS-MMJ}. Both statements,
rationality and discreteness, follows from the fact  that the set of
$F$-jumping numbers consists in rational numbers and it has no accumulation points for $F$-finite regular
domains \cite[Theorem B]{ST-NonPrincipal}.
\end{proof}

We now start relating the Bernstein-Sato polynomial with $F$-thresholds for direct summands. For this we need to establish a lemma that allow us to restrict $\QQ$-linear differential operators to $\ZZ$-linear operators.

\begin{lemma}\label{LemmaClearCoef}
Let $S=\QQ[x_1,\ldots,x_n]$, $I\subseteq S$ be an ideal, and $R=S/I$.
Let $S_{\ZZ}=\ZZ[x_1,\ldots,x_n]$, $J=I\cap S_{\ZZ},$ and $R_{\ZZ}=S_{\ZZ}/J$.
Then, for every $\delta\in D_{R|\QQ},$ there exists $c\in\NN$ such that
$c\delta(R_{\ZZ})\subseteq R_{\ZZ}.$ Furthermore, $c\delta \in D_{R|\ZZ}.$
\end{lemma}
\begin{proof}
Recall that
$$
D_{R|\QQ}=\frac{\{\partial \in D_{S|\QQ} \; |\; \partial(I)\subseteq I \}}{ID_{S|\QQ}}.
$$

Given $\delta \in D_{R|\QQ}$, let $\partial \in D_{S|\QQ}$ be such that
$\delta f=\overline{\partial f}$ for every $\overline{f}\in R.$ Since
$$
D_{S|\QQ}=S\left\langle \frac{d}{d x_1},\ldots,\frac{d}{d x_n}\right\rangle,
$$
there exists an integer $c$ such that
$$
c\partial\in S_{\ZZ}\left\langle\frac{d}{d x_1},\ldots,\frac{d}{d x_n}\right\rangle.
$$
Furthermore, $c\partial (S_{\ZZ})\subseteq S_{\ZZ}$ and $c\partial (J)\subseteq
J.$ It follows that $c\delta f=\overline{c\partial f}\in R_{\ZZ}$ for
every $\overline{f}\in R_{\ZZ}.$  Moreover, $c\delta\in D_{R|\ZZ}$ as
$c\partial\in D_{S_{\ZZ}|\ZZ}.$
\end{proof}

\begin{remark}
Let $R$ be a finitely generated $\ZZ$-algebra. Then, there is a ring
morphism $D_{R|\ZZ}\to D_{R\otimes_\ZZ\FF_p|\FF_p},$ defined by
$\delta\mapsto\overline{\delta},$  where
$\overline{\delta} \cdot \overline{f}=\overline{\delta \cdot f}.$ Furthermore,
this map  sends $D^n_{R|\ZZ}$ to $
D^n_{R\otimes_\ZZ\FF_p|\FF_p}.$  We point out that this map may not be surjective \cite[Pages 384-385 ]{DModFSplit}.
\end{remark}

We now relate the Bernstein-Sato polynomial to $F$-thresholds.

\begin{theorem}\label{ThmBSFThresholds}
Let $S=\QQ[x_1,\ldots,x_n]$, $I\subseteq S$ be an ideal, and
$R=S/I$. Let $S_{\ZZ}=\ZZ[x_1,\ldots,x_n]$, $J=I\cap S_{\ZZ},$ and
$R_{\ZZ}=S_{\ZZ}/J$. Suppose that $b^R_f(s)$ exists. Then, there
exists $m\in\NN$, depending only on $R_{\ZZ}$, such that
$$
b^R_f(\nu^\a_{\overline{f}}(p^e))\equiv 0\quad \hbox{mod }p
$$
for every prime number $p>m,$ and every ideal $\a\subseteq
R_{\ZZ}/pR_{\ZZ}$ such that $\overline{f}\in\sqrt{\a}$. In
particular, this holds when $R$ is a direct summand of a polynomial
ring over $\QQ.$
\end{theorem}
\begin{proof}
We pick a differential operator $\delta(s)=\sum^m_{j=1}\delta_j s^j\in D_{R|\QQ}[s]$
such that
\begin{equation}\label{BSclean}
\delta(t)f^{t+1}=b_f(t) f^{t}
\end{equation}
for every $t\in\ZZ.$
By Lemma \ref{LemmaClearCoef} we may clean up the coefficients of $\delta(s)$ so,
without any loss of generality, we may assume that $\delta(s)\in D_{R|\ZZ}[s]$ as long as we pick $p>c$, where $c$ is as in Lemma \ref{LemmaClearCoef}.
Moreover, me may multiply both sides of Equation \ref{BSclean} by the least common multiple of
the denominators of every coefficient of $b^R_f(s)$.
We pick $p$ bigger than any of this denominators.
We may also assume that it is an
equation in $R_{\ZZ}$.

Now we reduce this equation mod $p$.
By our choice of $p,$ $b^R_f(s)$ is not the zero polynomial in $\FF_p[s].$ We have the equation
$$\overline{\delta(t)}\overline{f}^{t+1}=\overline{b_f(t)}\overline{f}^{t}$$
for every $t\in\NN.$
We pick
$p>\max{n_j},$ where $n_j$ denotes the order of the differential operator $\delta_j$. Then,
the order of $\overline{\delta(t)}$ is bounded above by $p-1,$ so
 $\overline{\delta(t)}\in D_R^{(1)}=\Hom_{R}(F_* R ,F_* R)$  for every $t\in\NN$.
Let $\a\subseteq R_{\ZZ}/pR_{\ZZ}$ such that
$\overline{f}\in\sqrt{\a}$. We set $\nu=\nu^{\a}_f(p^e),$ thus we have
that $\overline{f}^{\nu+1}\in \a^{[p^e]}$ and
$\overline{\delta(\nu+1)} \a^{[p^e]}\subseteq \a^{[p^e]}.$ Then,
$$
b_f(\nu) \overline{f}^{\nu}=\overline{\delta(\nu)}\overline{f}^{\nu+1}\in \a^{[p^e]}.
$$
Since $\overline{f}\not\in I^{[p^e]},$ we deduce that
$b_f(\nu) \equiv 0$ mod $p$.
\end{proof}

We also point out that $\nu^\a_f(p^e)$ is the $e$-truncation of the base $p$-expansion of $c^\a(f)$.
For a similar result for $F$-pure thresholds, we refer to \cite[Key Lemma]{DanielMRL}.

Recall that the set of $F$-thresholds may differ from the set of
$F$-jumping numbers for non-regular rings. Our next result shows
that we still have an analogous relation between Bernstein-Sato
polynomials and  $F$-jumping numbers.

{
\begin{theorem}\label{ThmBSFjumps}
Let $S=\QQ[x_1,\ldots,x_n]$, $I\subseteq S$ be an ideal, and
$R=S/I$. Let $S_{\ZZ}=\ZZ[x_1,\ldots,x_n]$, $J=I\cap S_{\ZZ},$ and
$R_{\ZZ}=S_{\ZZ}/J$. Suppose that $b^R_f(s)$ exists. Then, there
exists $m\in\NN$, depending only on $R_{\ZZ}$,
$$
b^R_f(\nu)\equiv 0\quad \hbox{mod }p
$$
if  $p>m$ and $C^e_{R_{\ZZ}\otimes_\ZZ \FF_p } \overline{f}^\nu\neq
C^e_{R_{\ZZ}\otimes_\ZZ \FF_p } \overline{f}^{\nu+1}.$ In particular,
this holds when $R$ is a direct summand of a polynomial ring over
$\QQ.$
\end{theorem}
\begin{proof}
As in the proof of Theorem \ref{ThmBSFThresholds}, for $p\gg 0$,
we can assume that
$\delta(s)\in D_{R|\ZZ}[s]$. We can assume that no denominator in any coefficient in  $b^R_f(s)$ becomes zero.
Then,  $b^R_f(s)$ is not the zero polynomial in $\FF_p[s].$ We have the equation
$$\overline{\delta(t)}\overline{f}^{t+1}=\overline{b_f(t)}\overline{f}^{t}$$
for every $t\in\NN.$ We can also assume that
$\overline{\delta(t)}\in D_R^{(1)}=\Hom_{R}(F_* R,F_* R)$  for
every $t\in\NN$.

If $b_f(\nu)\not\equiv 0$ mod $p$, $\overline{f}^\nu\in D^{(1)}
\overline{f}^{\nu+1}.$ Then,
$$
D^{(e)}  \overline{f}^{\nu}\subseteq D^{(e)} D^{(1)} \overline{f}^{\nu+1}=D^{(e)}\overline{f}^{\nu+1}.
$$
Since $D^{(e)}  \overline{f}^{a+1}\subseteq D^{(e)}
\overline{f}^{\nu}$, we conclude that $D^{(e)}  \overline{f}^{\nu}=
D^{(e)}  \overline{f}^{\nu+1}$. Then, $C^e_{R_{\ZZ}\otimes_\ZZ \FF_p }
\overline{f}^\nu=C^e_{R_{\ZZ}\otimes_\ZZ \FF_p } \overline{f}^{\nu+1}$
by Remark \ref{LemmaRestrictionDmodCartier}, which contradicts our
hypothesis. Hence,  $b_f(\nu)\equiv 0$ mod $p$.
\end{proof}
}
\begin{remark}
Theorem \ref{ThmBSFThresholds} and \ref{ThmBSFjumps} also hold for
ideals that are not principal  if we consider the Bernstein-Sato
polynomial for arbitrary ideals (see Remark \ref{RmkGralBS}). These
claims can be also extended for any field of characteristic zero. We
focused on principal ideals over finitely generated $\QQ$-algebras
for the sake of  accessibility. The experts would easily adapt the
arguments in the proof to the more general setting.
\end{remark}

\section*{Acknowledgments}
We thank Jack Jeffries for suggesting Examples \ref{ExampleDiffBS} and   \ref{ExampleNoBS}.   We also thank Jen-Chieh Hsiao,  Luis Narv\'aez-Macarro and Karl Schwede for several helpful  comments.
The third author is grateful to Karen E. Smith for inspiring conversations.

\bibliographystyle{alpha}
\bibliography{References}

\end{document}